\documentclass[11pt,a4paper,twoside,psfull,reqno]{amsart}

\usepackage[T2A]{fontenc}
\usepackage[utf8x]{inputenc}
\usepackage[toc,page]{appendix}
\usepackage[fit]{truncate}
\usepackage[margin=10pt,font={scriptsize,bf},textfont=bf,labelsep=endash]{caption}
\usepackage[top=1in,left=1in,right=1in,bottom=1in]{geometry}
\usepackage{amsmath,amsfonts,amscd,amssymb,latexsym,rotating,xcolor,pict2e,hyperref,truncate}

\theoremstyle{plain}
 
\newtheorem{thm}{Theorem}[section]
\newtheorem{lemma}[thm]{Lemma}

\theoremstyle{definition}
\newtheorem{definition}[thm]{Definition}

 \numberwithin{equation}{section}
\title[Discrete Fractional maximal operator in Lorentz spaces]{CHARACTERIZATION OF THE BOUNDEDNESS OF THE DISCRETE FRACTIONAL MAXIMAL OPERATOR BETWEEN WEIGHTED LORENTZ SEQUENCE SPACES}

\author{A.Gogatishvili, N.A.Bokayev, N.K.Kuzeubayeva}

\address{Amiran Gogatishvili, Institute of Mathematical of the  Czech Academy of Sciences, \v Zitn\'a~25, 115~67 Praha~1, Czech Republic}
\email{gogatish@math.cas.cz}
\address{Nurzhan Bokave,  L.N. Gumilyov Eurasian National University,  Satpaev Str. 2, 010008, Astana, Republic of Kazakhstan}
 \email{bokayev2011@yandex.ru, nurgul.kuzeubaeva@mail.ru }
 \address{Nurgul Kuzeubaev,  L.N. Gumilyov Eurasian National University,  Satpaev Str. 2, 010008, Astana, Republic of Kazakhstan}
 \email{nurgul.kuzeubaeva@mail.ru }
\begin{document}
\begin{abstract}
    
 In this paper, we give the complete characterization of the boundedness of the discrete fractional maximal operator $M_\gamma $, defined on $\mathbb{Z}^n$,  $n\in\mathbb{N}$,  $\gamma\in [0,n)$,  between the classical Lorentz sequence spaces $\lambda_p(v)$ and $\lambda_q(w)$, where $0<p,q<\infty$,  $\{v(m)\}$ and $\{w(m)\}$ are non-negative sequences defined on $\mathbb{N}$. We first obtain a sharp upper estimate for the nonincreasing rearrangement of the discrete fractional maximal operator $M_\gamma x$, with a supremal operator involving a discrete Hardy operator and the nonincreasing rearrangement of the sequence $\{x(k)\}$, and we show that this estimate is sharp by showing a lower estimate for special sequences. Using this result, we can reduce the characterization of the boundedness of the discrete fractional maximal operator $ M_\gamma$ between the discrete weighted Lorentz sequence spaces to the boundedness of the discrete supremum operator between weighted Lebesgue sequence spaces, restricted to nonincreasing sequences. This problem was investigated in the recent paper by the authors, which enables us to obtain the final characterization.
 \end{abstract} 



\maketitle

\section{Introduction}

We start with the definition of the decreasing rearrangement of the real sequences. We will follow the book [1] reformulating the definitions for the special case of the $\sigma$-finite measure space $(\mathbb{Z}^n, \mathcal{P}(\mathbb{Z}^n),m)$ where $m$ is the usual counting measure, which is defined as follows: for a finite set $A$, $m(A)$ is the number of its elements, and for infinite $A$ we have that $m(A) = \infty$.
\begin{definition}
 [Distribution function and decreasing rearrangement]\label{def:rearrangement} 

 Consider the measure space $({\mathbb Z}^n,\mathcal{P}(\mathbb{Z}^n),m)$. For any  bounded sequence $\{x(k)\}_{k\in{\mathbb Z}^n} $ of real numbers,  we define its distribution function $m_x: [0,\infty) \rightarrow [0,\infty]$ in the following way:
\[ m_x(\omega) = m\{k \in \mathbb{Z}^n: |x(k)| > \omega\}, \quad \omega>0. 
\]
\end{definition}
Then for $t \ge 0$ we define the decreasing rearrangement:
\[ x^*(t) = \inf\{\omega > 0: m_x(\omega) \le t\} \]
and the decreasing rearrangement $\{x^*(i)\}_{i=1}^\infty$ of the sequence 
$\{x(i)\}_{k\in \mathbb{Z}^n}$ as
\[ x^*(i) = x^*(t(i)) = \inf\{\omega > 0: m_x(\omega) \le t(i)\} \]
for any $i-1 \le t(i) < i, \quad i\in \mathbb{N}$.

It should be noticed that an alternative equivalent definition of the decreasing rearrangement is possible.

\[x^*(i)=\sup_{E\subset\mathbb{Z}^n, m(E)=i}\min_{k\in E}|x(k)|\quad i\in \mathbb{N}.\]

We use the convention here that $\inf\emptyset=\infty$. Thus, if $m_x(\omega)>i$ for all $\omega \ge 0$, then $x^*(i)=\infty$.

We denote  $x^{**}(k)$ as the average $
\{x^*(k)\}_{k\in\mathbb{N}}$:  
\[
x^{**}(k) = \frac{1}{k} \sum_{i=1}^k x^*(i),\quad k\in  \mathbb{N}.
\]
We will use the following conventions. We write $A \lesssim B$ if there exists a constant $C\in (0, \infty)$ independent of appropriate quantities such that $A \leq C B$. We write $A \approx B$ if we have both $A \lesssim B $ and $B \lesssim A$.  We put $1/ (\pm \infty) =0$, $0/0 =0$, $0\cdot (\pm \infty)=0$. Moreover, we omit arguments of integrands as well as differentials in integrals when appropriate in order to keep the exposition as short as possible. We will denote by $\ LHS(*)$ and $\ RHS(*)$, the left-hand side and right-hand side of an inequality $(*)$, respectively.
Note that since the counting measure $m$ can only take integer values (or infinity), it does not matter which exact $t_k$ is chosen for the given $k$, as long as it is in the aforementioned interval.

In this paper, an upper bound is obtained for a nonincreasing rearrangement of a discrete fractional-maximal operator through a certain supremal operator.    
 We assume throughout that $\textbf{w}$ and $\textbf{v}$ are sequences of nonnegative terms. Their partial sums are denoted by
\[
W(m) =\sum^m_{k=1}w(k), \quad V(m) =\sum^m_{k=1}v(k), \quad m\in {\mathbb N}.
\] 
In the case where the series converges, or we consider the limit of the sequence of partial sums, we write:
\[
V(\infty) = \lim_{m \to \infty} V(m).
\]
\begin{definition} 
Given a fixed sequence $w(k)_{k\in\mathbb{N}}$ of nonnegative real numbers not identically equal to zero and a fixed $p \in (0,\infty)$.
The classical Lorentz sequence space
$ \lambda_p(w)$  is the set
of all complex sequences $x = {x(k)}_{k\in \mathbb{Z}^n}$ such that $\lVert x \rVert_{\lambda_p(w)} < \infty $,
where 

\[ \lVert x \rVert_{\lambda_p(w)} = \begin{cases}    
\left(\sum\limits_{k=1}^\infty (x^*(k))^p \cdot w(k)\right)^{\frac{1}{p}},\quad &0<p<\infty,\\
\sup\limits_{k\ge 1} x^*(k)\cdot w(k), \quad &p=\infty.
\end{cases}\]
 Notice that the $\|\cdot\|_{\lambda_p(w)}$ does not necessarily have to be a norm. 
\end{definition}
\begin{definition}

\[ \lambda_p(w) = \{x \in \mathbb{Z}^n: \, \lVert x \rVert_{\lambda_p(w)} < +\infty\}. \]
\end{definition}
For  $m=(m_1, m_2, \cdots , m_n) \in \mathbb{Z}^n$   and   $N\in \mathbb{N}$ we define the set
\begin{equation*}
S_{m,N}:=\{k\in\mathbb{Z}^n: ||k-m||_\infty \leq N\}, \quad 
\end{equation*}
where $||(x(1),x(2),\cdots, x(n))||_\infty:=\max\{|x(i)|: 1\leq i \leq n\}$
for each $(x(1),x(2), \cdots,x(n))\in\mathbb{Z}^n$. Then  $| S_{m,N}|=(2N+1)^n$ is the cardinality of the set $S_{m,N}$.
Let $\gamma \in [0, n)$, 
$x=(x(1),x(2),\cdots,x(k))\in\mathbb{Z}^n$. The discrete fractional maximal operator $M_\gamma$ is defined by

\begin{equation*} 
(M_{\gamma}x)(m)=\sup_{N\in\mathbb{N}}{{|S_{m,N}|^{\frac{\gamma}{n}-1}}}\displaystyle \sum\limits_{k\in S_{m,N}} |x(k)|, \,\, m\in \mathbb{Z}^n.
\end{equation*}

The aim of this paper is to establish a condition for the boundedness of the discrete fractional maximal operator from $\Lambda_p(v)$ to $\Lambda_q(w)$.

\section{Preliminaries}
\begin{thm}{\cite[Chapter-2, Theorem~2.2]{BS}}\label{HL}
If $x=\{x(m)\}_{m\in\mathbb{Z}^n}$ and $b=\{b(m)\}_{m\in\mathbb{Z}^n}$ are complex sequences, then
\[ \sum_{m\in\mathbb{Z}^n} x(m)b(m) \le \sum_{k=1}^\infty x^*(k)b^*(k). \]
\end{thm}

\begin{lemma}
Let $n \in \mathbb{N}$ and $\gamma \in [0, n)$, then the following inequalities hold:
\begin{align} \sup_{k\in \ \mathbb{N}} (M_{\gamma}x)^*(k) &\leq D_3 \sup_{k \in \ \mathbb{N}} k^{\frac{\gamma}{n}} x^{*}(k),    \label{n gamma}\\
\sup_{k\in \ \mathbb{N}} k^{1-\frac{\gamma}{n}} (M_{\gamma}x)^*(k) &\leq D_4 \sum_{k\in \mathbb{Z}^n} {|x(k)|},   \label{1, n(n-gamma)} 
\end{align}

where $D_3$ and $D_4$ are positive constants depending only on $n$ and $\gamma$. 
\end{lemma}

\begin{proof} 
Using the definition of $M_\gamma$ and Theorem~\ref{HL} we get 

\begin{align*}
  \sup_{k\in\ \mathbb{N}}(M_{\gamma}x)^*(k) &=\sup_{m\in\mathbb{Z}^n}(M_{\gamma}x)(m)\\
  &=\sup_{m\in\mathbb{Z}^n}\sup_{N\in\mathbb{N}}{{|S_{m,N}|^{\frac{\gamma}{n}-1}}}\displaystyle\sum\limits_{k\in S_{m,N}} |x(k)|\\
  &\le\sup_{m\in\mathbb{Z}^n}\sup_{N\in\mathbb{N}}{{|S_{m,N}|^{\frac{\gamma}{n}-1}}}\displaystyle\sum\limits_{i=1}^{|S_{m,N}|} x^*(i)\\
  &= \sup_{m\in\ \mathbb{Z}^n}\sup_{N\in\mathbb{N}}{{|S_{m,N}|^{\frac{\gamma}{n}-1}}}\displaystyle\sum\limits_{i=1}^{|S_{m,N}|}i^{-\gamma/n} i^{\gamma/n}x^*(i)\\
  &\le\sup_{m\in\mathbb{Z}^n}\sup_{N\in\mathbb{N}}{{|S_{m,N}|^{\frac{\gamma}{n}-1}}}\displaystyle\sum\limits_{i=1}^{|S_{m,N}|}i^{-\gamma/n}\sup_{i\in\mathbb{N}}i^{\gamma/n}x^*(i)\\
  & \le D_3\sup_{i\in\mathbb{N}}i^{\gamma/n}x^*(i)
\end{align*}

the estimate \eqref{n gamma} proved.

Now  we prove the estimate \eqref{1, n(n-gamma)}. 
Fix a number $\lambda>0$ and set

\[
E_\lambda = \left\{ k \in \mathbb{Z}^n : (M_\gamma x)(k) > \lambda \right\}.
\]
We want to estimate
\[
m(E_\lambda)^{1-\frac{\gamma}{n}} = m(\left\{ k \in \mathbb{Z}^n : (M_\gamma x)(k) > \lambda \right\})\le \frac{\widetilde{C}_4}{\lambda}\sum_{k\in \mathbb{Z}^n} |x(k)|.
\]
For each $ k \in E_\lambda $, choose  $N_k $ such that
\[
\frac{1}{|S_{k,N_k}|^{1 - \gamma/n}} \sum_{m\in S_{k,N_k}} |x(m)| > \lambda.
\]\
Thus \( E_\lambda \subset \cup_{k\in E_\lambda} S_{k, N_k} \).
Take any finite set $E\subset E_\lambda$, and apply the Vitali covering lemma 
{\cite[Chapter-1, Lemma, p.~9]{St}}. So there exists a disjoint subcollection  \( \{ S_j=S_{k_j, N_{k_j}} \} \) selected by the Vitali covering lemma, such that:
\[m(E_\lambda)\le 3^n \sum_j m(S_j)
\]
and
\begin{align*}
m(E_\lambda)^{1-\frac{\gamma}{n}}&\le 3^{n-\gamma} \sum_j m(S_j)^{1-\frac{\gamma}{n}}\notag\\
&\le 
 3^{n-\gamma}\lambda^{-1}\sum_j \sum_{m \in S_j} |x(m)| \notag \\
 &\le 
 3^{n-\gamma}\lambda^{-1}\sum_{m\in \cup_j S_j} |x(m)|\notag\\
   & \le 
 3^{n-\gamma}\lambda^{-1}\sum_{m\in \mathbb{Z}^n} |x(m)|,
 \end{align*}
so
\[
\sup_{\lambda>0}\lambda m(E_\lambda)^{1-\frac{\gamma}{n}} \leq 3^{{n-\gamma}} \sum_{m\in \mathbb{Z}^n} |x(m)|
\]

\[\sup_{m\in \mathbb{N}} m^{1-\frac{\gamma}{n}} (M_{\gamma}x)^*(m)=\sup_{\lambda>0}\lambda m(E_\lambda)^{1-\frac{\gamma}{n}}\leq 3^{{n-\gamma}} \sum_{m\in \mathbb{Z}^n} |x(m)|\]
which is an estimate \eqref{1, n(n-gamma)} with  $D_4=3^{n-\gamma}$.
\end{proof}











\section{Main results}

In the following theorem, an upper bound estimate for the nonincreasing rearrangement of the discrete fractional maximal operator is obtained.

Having defined the $\lambda_p(w)$ norm, we can now move on to the sequence spaces induced by that norm. For the definition to make sense, we only take sequences with finite norm.

\begin{thm} \label{th4}
 Let $n \in \mathbb{N}$ and $\gamma \in [0, n)$. Then there exists a positive constant $D_1$, depending only on $n$ and $\gamma$, such that
\begin{equation} \label{eq1}
(M_\gamma x)^*(m) \leq D_1 \sup_{m\leq k<\infty} \ k^{\gamma/n} x^{**}(\ k), \quad m \in  \mathbb{N}     \end{equation}
for any sequence $x=\{x(m)\}_{m\in\mathbb{Z}^n}$. 
\end{thm}

\begin{proof}[{\bf Proof of Theorem~\ref{th4}.}]

 Fix $k\in \mathbb{N}$ and let $x\in l^1_{loc} (\mathbb{Z}^n)$. We may assume that

\begin{equation}
  \sup_{m\leq k<\infty}k^{\gamma/n-1}\sum_{j=1}^{k}x^{*}(j)< \infty    \label{eq6}\end{equation}

otherwise \eqref{eq1} holds trivially. Then by the Hardy-Littlewood inequality \eqref{HL} 

$$\sum\limits_{k\in Q}|x(k)|\leq \sum\limits^{k}_{i=1} x^{*}(i)< \infty$$
for every set $Q\subset \mathbb{Z}^n$ of measure at most $k$. In particular, if we put

$$E=\{m\in \mathbb{Z}^n: |x(m)|>x^*(k)\},$$

then $m(E)\leq k$ \cite[Chapter 2, (1.18)]{BS}, and so $\sum\limits_{k\in E}|x(k)|<\infty$. We define the sequences:

\begin{align} 
y_{k}(m)&= \max\{{|x(m)|-x^*(k),0}\}\operatorname{sgn}{x(m)},\;\;&m\in\mathbb{Z}^n,\label{yk}\\
z_{k}(m)&= \min\{{|x(m)|, x^*(k)}\}\operatorname{sgn}x(m),\;\;&m\in \mathbb{Z}^n.\label{zk}
\end{align}

Then $x=y_k+z_k$ and

\begin{align} 
y^*_{k}{(i)}&= \chi_{(0;k)\cap\mathbb{N}}{(x^*(i)-x^*(k))}, \;\; &i\in \mathbb{N}\label{yk*}\\
z^*_{k}{(i)}&= \min {\{x^*(i),x^*(k)\}}, \;\; &i\in \mathbb{N}\label{zk*}
\end{align}

whence from \eqref{yk*}:

\begin{equation}\sum_{m\in\mathbb{Z}^n}|y_k(m)|=\sum\limits^{\infty}_{i=1} y^{*}_{k}(i)=\sum\limits^{k}_{i=1} \Big( x^{*}(i)-x^{*}(k)\Big)\leq\sum\limits^{k}_{i=1} x^{*}(i).  \label{sumym}
\end{equation}

From the inequality \eqref{1, n(n-gamma)} it follows that

\begin{equation}(M_\gamma y_k)^*{(i)}\leq D_4{i^{\gamma/n-1}}\sum_{m\in \mathbb{Z}^n}|y_k(m)|\leq D_4{i^{\gamma/n-1}}\sum_{i=1}^k x^*(i), \;\; i\in  \mathbb{N}. \label{myk<}
\end{equation}

By \eqref{n gamma} and \eqref{zk*} we have:

\begin{eqnarray*}
  (M_\gamma z_k)^*{(i)} &\leq& D_3\sup_{1\le i<\infty}{i^{\gamma/n}}z_k^*{(i)}\\
  &=& D_3\max\Big\{\sup_{0<i<k}{i^{\gamma/n}} x^*{(k)}, \sup_{k\leq i<\infty}{i^{\gamma/n}} x^*{(i)}\Big\} \\
  &\approx& D_3\max\Big\{k^{\gamma/n}x^*(k), \sup_{k\leq i<\infty}{i^{\gamma/n}}x^*{(i)}\Big\} \\
  &\leq & D_3\sup_{k\leq i<\infty}{i^{\gamma/n}}x^*{(i)}.
\end{eqnarray*}

Hence

\begin{equation}\sup_{0<i<\infty}(M_\gamma z_k)^*{(i)}\leq D_3\sup_{k\leq i<\infty}{i^{\gamma/n}}x^*{(i)}, \label{mzk<} \end{equation}
\


$$\big(M_{\gamma}x\big)^*{(k)}\leq \big(M_{\gamma}y_k\big)^*{({k})}+\sup_{0<i<\infty}\big(M_{\gamma}z_k\big)^*{({i})}.$$

By using, \eqref{myk<} and \eqref{mzk<} we get

\begin{eqnarray*}
\big(M_{\gamma}x\big)^*{(k)}&\leq& D_4{k^{\gamma/n-1}}\sum_{i=1}^k x^*(i)+ D_3\sup_{k\le i<\infty} i^{\gamma/n}x^*{(i)} \\
&\leq& (D_4+D_3)\sup_{k\le i<\infty}{i^{\gamma/n}}x^{**}{(i)}\\
&=&D_1\sup_{k\le i<\infty}{i^{\gamma/n}}x^{**}{(i)},
\end{eqnarray*}
with $D_1=D_3+D_4$.
\end{proof}

\begin{thm} \label{th5}
Let  us $n \in \mathbb{N}$ and $\gamma \in [0, n)$. Inequality \eqref{eq1} in Theorem~\ref{th4}  is sharp in
the sense that for any positive nonincreasing sequence $x=\{x(m)\}_{m\in \mathbb{N}}$ there
exists a sequence  $y=\{y(k)\}_{k\in\mathbb{Z}^n}$ such that  $ y^*(k)=x(k)$, $k\in \mathbb{N}$ and the inequality
holds:
\begin{align} \label{eq2}
(M_\gamma y)^*(m) &\geq D_2 \sup_{m\leq k<\infty} \ k^{\gamma/n} y^{**}(\ k)\notag\\
&= D_2\sup_{m\leq k<\infty} k^{\frac{\gamma}{n}-1}\sum^k_{i=1} x(i), \quad {m, k\in \mathbb{N}},   
\end{align}

  where $D_2$ - is a positive constant depending only on $n$ and  $\gamma$. 
  \end{thm}
  
\begin{proof}[\bf Proof of Theorem~\ref{th5}]
 
  Let for any   ${m\in \mathbb{N}}$ define

\begin{align*}
N_1&=\{1\},\\
N_{m}&:=\{i\in\mathbb{N}: i=(2m-3)^n+1, \cdots,  (2m-1)^n \}, \quad
m=2,3, \dots,
\end{align*} 

and

\begin{align*}
Q_{m}:=\{\bar{k}\in \mathbb{Z}^n: \max_{1\le i\le n} |k(i)|=m-1\}.
\end{align*}

It is clear that 
\begin{align*}\#\{N_1\}&=\#\{Q_1\}=1, \\
\intertext {and} \#\{N_m\}&=\#\{Q_m\} = (2m-1)^n-(2m-3)^n, \quad m=2,3,\cdots
\end{align*}
$$\mathbb{N}=\cup_{m\in\mathbb{N}}N_m, \quad \mathbb{Z}^n=\cup_{m\in\mathbb{N}}Q_m.$$
Let 
\[ \Theta_m: Q_m\to N_m, \quad m\in \mathbb{N}\]
be one-to-one functions.

Define $\Theta:\mathbb{Z}^n\to \mathbb{N}$
\[\Theta(k)=\Theta_m(k), \quad \text{if}\quad  k\in Q_m. \]
\[y(k)=x(\Theta(k)),\quad k\in \mathbb{Z}^n.\]
Then, for every $k,m\in \mathbb{Z}^n$ such that $\|k\|_\infty\ge \|m\|_\infty$, we have
\begin{align*} M_{\gamma}(y)(m)&\ge c|S_{0,\|k\|_\infty}|^{\frac{\gamma}{n}-1}\sum_{j\in S_{0,\|k\|_\infty}}|y(j)|\\
&=c\bigg(2\|k\|_\infty +1\bigg)^{\gamma -n}\sum_{i=1}^{\|k\|_\infty}\sum_{j\in Q_i}|y(j)|\\
&=c\bigg(2\|k\|_\infty +1\bigg)^{\gamma -n}\sum_{i=1}^{\|k\|_\infty}\sum_{j\in Q_i}|x(\Theta(j))|\\
&=c\bigg(2\|k\|_\infty +1\bigg)^{\gamma -n}\sum_{i=1}^{\|k\|_\infty} \sum_{m\in N_i}|x(m)|\\
&=c\bigg(2\|k\|_\infty +1\bigg)^{\gamma -n}\sum_{i=1}^{(2\|k\|_\infty+1)^n} |x(i)|.
\end{align*}
Consequently, 
\[ M_{\gamma}(y)(m)\ge c \sup_{j\ge (2\|m\|_\infty+1)^n} j^{\frac{\gamma}{n} -1}\sum_{i=1}^{j} |x(i)|.\]
Hence, \eqref{eq2} followed on taking rearrangements. 
\end{proof}

\textbf{Remark 1.} For the classical fractional maximal operator, similar statements are proven in \cite{CKOP}, and for the generalized fractional maximal operator, they are considered in \cite{BGA}. \\

From Theorem~\ref{th4} and Theorem~\ref{th5} follows  the following.

\begin{lemma}\label{th6} 
Let $n\in\mathbb{N}$,  $\gamma\in [0,n)$ and $0<q,p\leq\infty$. Assume that $\{v(m)\}$ and $\{w(m)\}$ are given nonnegative weight sequences.
Then the operator $M_{\gamma}$ is bounded from $\lambda_p(v)$ to $\lambda_q(w)$ ($M_{\gamma}: \lambda_p(v)\hookrightarrow \lambda_q(w)$) i.e.
\begin{equation}  \label{eq3}
\left(\sum_{k=1}^{\infty}
\left((M_\gamma y)^*(k)\right)^q w(k)\right)^{\frac{1}{q}} \le\widetilde{C}_1
\left(\sum^\infty_{k=1} (y(k)^*)^p v(k)\right)^{\frac{1}{p}}
\end{equation}
holds for any sequence $y=\{y(m)\}_{m\in\mathbb{Z}^n}$,
if and only if the following inequality 
\begin{equation}  \label{eq4}
\left(\sum_{m=1}^{\infty}
\left(\sup_{m<k} k^{\frac{\gamma}{n}-1}\sum^k_{i=1}x(i)\right)^q w(m)\right)^{\frac{1}{q}} \le \widetilde{C}_2
\left(\sum^\infty_{k=1} x(k) ^p v(k)\right)^{\frac{1}{p}}
\end{equation}
 holds  for all nonnegative, nonincreasing sequences $\{x(n)\}_{n\in \mathbb{N}}$. 
     Moreover, if   $\widetilde{C}_1$  and $\widetilde{C}_2$  are the best constants in~\eqref{eq3} and  ~\eqref{eq4}, respectively, then
$\widetilde{C}_1 \approx \widetilde{C}_2$ and the multiplicative constants in all the equivalences above depend only on $p,q$.
\end{lemma}
\begin{proof}[{\bf Proof of Lemma~\ref{th6}}]
We will prove the equivalence of two statements:

    The boundedness of the operator $M_\gamma$ from $\lambda_p(v)$ to $\lambda_q(w)$, i.e., the fulfillment of inequality \eqref{eq3}.
    The fulfillment of inequality \eqref{eq4} for all nonnegative, nonincreasing sequences $\{x(n)\}_{n \in \mathbb{N}}$.

 $\rm \mathbf{\eqref{eq3} \Rightarrow \eqref{eq4}}.$

Assume that the operator $M_\gamma$ is bounded, i.e. the inequality \eqref{eq3} holds for any sequence $\{y(m)\}_{m\in \mathbb{Z}_n}$. We will show that this implies inequality \eqref{eq4} for nonincreasing sequences $x(k)_{k\in \mathbb{N}}$.
By Theorem~\ref{th5} for given nonincreasing sequence $(x(k)_{k\in \mathbb{N}}$ there exists multivariable sequence  $\{y(m)\}_{m\in \mathbb{Z}_n}$ such that $y^*(k)=x(k)$ $k\in \mathbb{N}$ and

\[ (M_{\gamma}(y))^*(k) \ge D_2 \sup_{k\le i< \infty} i^{\frac{\gamma}{n}-1}\sum_{j=1}^ix(j).
\]

Therefore, from inequality \eqref{eq3} follows 

\begin{align*} &\left(\sum_{k=1}^{\infty} \left(\sup_{k< i<\infty} i^{\frac{\gamma}{n} -1} \sum_{j=1}^i x(j)\right)^q w(k)\right)^{1/q}
\\
&\hskip+1cm\le D_2^{-1} \left(\sum_{k=1}^{\infty} (M_\gamma(y))^*(k))^q w(k)\right)^{\frac{1}{q}}\le D_{2}^{-1}\widetilde{C}_1  \left(\sum_{k=1}^{\infty} y^*(k)^p v(k)\right)^{\frac{1}{p}}\\
 &\hskip+1cm=D_{2}^{-1}\widetilde{C}_1  \left(\sum_{k=1}^{\infty} x(k)^p v(k)\right)^{\frac{1}{p}}.
\end{align*}
Therefore, we obtain  \eqref{eq4} with 
\[\widetilde{C}_2\le  D_{2}^{-1}\widetilde{C}_1.\]
 {\bf \eqref{eq4} $\Rightarrow$ \eqref{eq3}.}

Assume that inequality \eqref{eq4} holds for all nonnegative, nonincreasing sequences $\{x(n)\}$. We will show that this implies the boundedness of the operator $M_\gamma$, i.e., inequality \eqref{eq3} for any sequence $y$.

Using Theorem~\ref{th4} and \eqref{eq4} for nonincreasing sequence \newline $((y^*)(k))_{k\in \mathbb{N}}$, we obtain 
\begin{align*}  &\left(\sum_{k=1}^{\infty} ((M_{\gamma} y)^*(k))^qw(k)\right)^{\frac{1}{q}}\\
&\hskip+2cm\le D_1 
 \left(\sum_{k=1}^{\infty} \left(\sup_{k< i<\infty} i^{\frac{\gamma}{n}-1} \sum_{j=1}^i y^*(j)\right)^q w(k)\right)^{\frac{1}{q}}\\
 & \hskip+2cm\le D_1\widetilde{C}_2  \left(\sum_{k=1}^{\infty} \left( y^*(k)\right)^p v(k)\right)^{\frac{1}{p}}.
 \end{align*}
 Therefore, we obtain  \eqref{eq3} with 
\[\widetilde{C}_1\le  D_{1}\widetilde{C}_2.\]
\end{proof}

In the theorem below, we give a characterization for computing inequality \eqref{eq3} under various choices of the parameters p and q.

\begin{thm}   
    \label{thmmain}
Let $n\in \mathbb{N}$. $\gamma\in [0,n)$, and $0<p, q<\infty$. Assume that $\{w(m)\}$ and   $\{v(m)\}$  are given
nonnegative real numbers not identically equal to zero.  
Then the inequality 
\begin{equation*}  
\left(\sum_{k=1}^{\infty}
\left((M_\gamma y)^*(k)\right)^q w(k)\right)^{\frac{1}{q}} \le\widetilde{C}_1
\left(\sum^\infty_{k=1} (y(k)^*)^p v(k)\right)^{\frac{1}{p}}
\end{equation*}

holds  for any sequences $\{x(m)\}$ if and only if: 

{\rm (i)}				If $1 < p\le q$, then $\max\{C_0, C_1, C_2, C_3\} < \infty$, where
				\begin{align*}C_0& = \Bigg(\sum_{i=1}^\infty i^{\frac{\gamma q}{n}} w(i)\Bigg)^{\frac 1q}\Bigg(\sum_{i=1}^\infty v(i)\Bigg)^{-\frac{1}{p}},\\
                C_1& = \sup_{k\in\mathbb{N}} k^{\frac{\gamma}{n}-1}W(k)^{\frac 1q}                    \Bigg(\sum_{m=1}^k\frac{V(m)^{\frac{1}{1-p}} v(m+1)}{ m^{\frac{p}{1-p}} V(m+1)}\Bigg)^{\frac{p-1}{p}},\\
				    		C_2& = \sup_{k\in\mathbb{N}} \Bigg( \sum_{i= k}^\infty  i^{q(\frac{\gamma}{n}-1)} w(i) \Bigg)^{\frac 1q}\Bigg(\sum_{m=1}^k\frac{V(m)^{\frac{1}{1-p}} v(m+1)}{ m^{\frac{p}{1-p}} V(m+1)}\Bigg)^{\frac{p-1}{p}},\\
									C_3 &= \sup_{k\in \mathbb{N}}  \Bigg(\sum_{i=1} ^ k  w(i) i^{\frac{q\gamma}{n}} \Bigg)^{\frac 1q}
                                    V(k)^{-\frac{1}{p}}.
			\end{align*}
{\rm (ii)}
				If $p > 1$ and $q<p$, then $\max\{C_0,C_4, C_5 ,C_6 ,C_7\} < \infty$, where
					\begin{align*}
						C_4 & =\Bigg( \sum_{k=1}^\infty   \Bigg(\sum_{i=k}^\infty i^{q(\frac{\gamma}{n}-1)} w(i)\Bigg)^{\frac{q}{p-q}} k^{q(\frac{\gamma}{n}-1)} w(k)\Bigg(\sum_{m=1}^k\frac{V(m)^{\frac{1}{1-p}} v(m+1)}{ m^{\frac{p}{1-p}} V(m+1)}\Bigg)\Bigg)^{\frac{p-1}{p}},\\
			                     C_5&= \Bigg( \sum_{k=1}^\infty   W(k)^{\frac{q}{p-q}} w(k) \sup_{k\le i<\infty}i^{(\frac{\gamma}{n}-1)\frac{pq}{p-q}}\Bigg(\sum_{m=1}^k\frac{V(m)^{\frac{1}{1-p}} v(m+1)}{ m^{\frac{p}{1-p}} V(m+1)}\Bigg)^{\frac{(p-1)q}{p-q}} \Bigg)^\frac{p-q}{pq},\\
                            C_6 & = \Bigg( \sum_{k=1}^\infty   W(k)^{\frac{q}{p-q}}w(k) \sup_{k\le i<\infty} i^{\frac{\gamma}{n}\frac{pq}{p-q}}V(i)^{-\frac{q}{p-q}}\\
                            &\hskip+2cm\times\Bigg(\sum_{m\ge i}V(m+1)^{-1} V(m)^{\frac{1}{1-p}} v(m+1)\Bigg)\sp{\frac{(p-1)q}{p-q}} \Bigg)^\frac{p-q}{pq}\\
						C_7&= \Bigg( \sum_{k=1}^\infty  \Bigg(\sum_{i=1}^k w(i)  i^{\frac{q\gamma}{n}} \Bigg)^{\frac{q}{p-q}} w(k) 
                       \sup_{k\le i<\infty}i^{\frac{\gamma q}{n}}V(i)^{-\frac{q}{p-q}} \Bigg)^\frac{p-q}{pq}.
					\end{align*}
            {\rm (iii)}				If $0<p\le 1$ and $p\le q$, then $\max\{C_8, C_9\} < \infty$, where
					\begin{align*}
						C_8 &= \sup_{k\in\mathbb{N}} k^{\frac{\gamma}{n}-1} W(k)^{\frac 1q}\sup_{1\le i\le k}i V(i)^{-\frac{1}{p}}.\\
                        C_9 &= \sup_{k\in\mathbb{N}} \Bigg(\sum_{i=k}^\infty i^{q(\frac{\gamma}{n}-1)} w(i) \Bigg)^{\frac 1q}\sup_{1\le i\le k}i V(i)^{-\frac{1}{p}}.\\
					\end{align*}
            {\rm (iv)}
				If $0<q<p\le 1$, then then $\max\{C_{10}, C_{11}\} < \infty$, where
					\begin{align*}	
                    C_{10} &=  \Bigg( \sum_{k=1}^\infty  \Bigg(\sum_{i=k}^\infty w(i) i^{q(\frac{\gamma}{n}-1)} \Bigg)^{\frac{q}{p-q}}	k^{q(\frac{\gamma}{n}-1)} w(k) \sup_{1\le i\le k}  i^{\frac{pq}{p-q}}V(i)^{\frac q{q-p}}\Bigg)\\
						C_{11} &=\Bigg(\sum_{k=1}^\infty  W(k)^{\frac{q}{p-q}} w(k) \sup_{k\le i<\infty}i^{\frac{\gamma}{n}\frac{pq}{p-q}} V(i)^{\frac{q}{q-p}} \Bigg)^{\frac{p-q}{pq}}.
                    \end{align*}

    Moreover, if we denote by $\widetilde{C}$ the best constant in~\eqref{eq4}, then
\begin{equation*}
   \widetilde{C} \approx
        \begin{cases}
           \max\{C_0, C_1, C_2, C_3\} &\text{in the case \textup{(i)},}
                \\
         \max\{C_0, C_4, C_5, C_6, C_7\}  &\text{in the case \textup{(ii)},}
                \\
        \max\{C_{8}, C_{9}\}  &\text{in the case \textup{(iii)},}
                \\
     \max\{C_{10}, C_{11}\}  &\text{in the case \textup{(iv)}.}
        \end{cases}
\end{equation*}
Moreover, the multiplicative constants in all the equivalences above depend only on $p,q$.
\end{thm}

\begin{proof}[{\bf Proof of Theorem~\ref{thmmain}.}]
The proof easily follows from Lemma~\ref{th6} by using \cite[Theorem~2.3]{UGBK}, with   $u(i)=i^{\frac{\gamma}{n}-1} $, $a(m)=w(m)$ and $b(m)=v(m)$.
It is easily seen that $z(k)=k^{\frac{\gamma}{n}-1}$ and  $\sup_{i\le j\le n}u(j)j=n^{\frac{\gamma}{n}}$. Therefore, in the cases (iii) and (iv), we have one condition less than in \cite[Theorem 2.3]{UGBK} because in case (iii) we have 
$C \approx C_8+C_9+C_{10}$ and $C_{10}\le C_8$. In case (iv)  we have $C \approx C_{11}+C_{12}+C_{13}$,  $C_{13}\le C_{12}$, therefor $C \approx C_{11}+C_{12}$, $C_{11}$ and $C_{12}$ are now $C_{10}$ and $C_{11}$.  

\end{proof}

We note that weighted inequalities for Hardy operators and for other operators have also been considered in the literature \cite{GKOP}-\cite{S}.

\section{Acknowledgements}
The research was supported by the grant Ministry of Education and Science of the Republic of Kazakhstan (project no. AP26196065).The research of  A.~Gogatishvili was partially supported by the grant project 23-04720S of the Czech Science Foundation (GA\v{C}R),  The Institute of Mathematics, CAS is supported by RVO:67985840, by  Shota Rustaveli National Science Foundation (SRNSF), grant no: FR22-17770.

\end{document}